\documentclass[11 pt]{amsart}
\usepackage{amscd,amsfonts,amssymb,amsmath}
\usepackage{hyperref}
\usepackage{epsfig}
\newtheorem{theorem}{Theorem}[section]
\newtheorem{corollary}[theorem]{Corollary}
\newtheorem{lemma}[theorem]{Lemma}
\newtheorem{proposition}[theorem]{Proposition}
\theoremstyle{definition}
\newtheorem{conjecture}[theorem]{Conjecture}
\newtheorem{question}[theorem]{Question}
\newtheorem{problem}[theorem]{Problem}

\newtheorem{remark}[theorem]{Remark}
\numberwithin{equation}{subsection}
\newtheorem*{ack}{Acknowledgement}
\usepackage[all,cmtip]{xy}
\usepackage{xcolor}

\newcommand{\Alex}{\operatorname{Alex}}
\newcommand{\Aut}{\operatorname{Aut}}

\newcommand{\Conj}{\operatorname{Conj}}
\newcommand{\Core}{\operatorname{Core}}

\newcommand{\cl}{\operatorname{cl}}
\newcommand{\Cs}{\operatorname{Cs}}
\newcommand{\cw}{\operatorname{cw}}

\newcommand{\Inn}{\operatorname{Inn}}

\newcommand{\mq}{\operatorname{mq}}
\newcommand{\T}{\operatorname{T}}

\newcommand{\R}{\operatorname{R}}

\newcommand{\id}{\mathrm{id}}

\begin{document}
\title{Zero-divisors and idempotents in quandle rings}
\author{Valeriy G. Bardakov}
\author{Inder Bir S. Passi}
\author{Mahender Singh}

\date{\today}
\address{Sobolev Institute of Mathematics and Novosibirsk State University, Novosibirsk 630090, Russia.}
\address{Novosibirsk State Agrarian University, Dobrolyubova street, 160, Novosibirsk, 630039, Russia.}
\email{bardakov@math.nsc.ru}

\address{Center for Advanced Study in Mathematics, Panjab University, Chandigarh 160014, India.}
\address{Department of Mathematical Sciences, Indian Institute of Science Education and Research (IISER) Mohali, Sector 81,  S. A. S. Nagar, P. O. Manauli, Punjab 140306, India.}
\email{passi@iisermohali.ac.in}

\address{Department of Mathematical Sciences, Indian Institute of Science Education and Research (IISER) Mohali, Sector 81,  S. A. S. Nagar, P. O. Manauli, Punjab 140306, India.}
\email{mahender@iisermohali.ac.in}

\subjclass[2010]{Primary 17D99; Secondary 57M27, 16S34, 20N02}
\keywords{Commutator width, idempotent,  Kaplansky's zero-divisor conjecture, semi-latin quandle, orderable quandle, quandle ring, zero-divisor}

\begin{abstract}
The paper develops further the theory of quandle rings which was introduced by the authors in a recent work. Orderability of quandles is defined and many interesting examples of orderable quandles are given. It is proved that quandle rings of left or right orderable quandles which are semi-latin have no zero-divisors. Idempotents in quandle rings of certain interesting quandles are computed and used to determine sets of maximal quandles in these rings. Understanding of idempotents is further applied to determine automorphism groups of these quandle rings. Also, commutator width of quandle rings is introduced and computed in a few cases.  The paper conclude by commenting on relation of quandle rings with other well-known non-associative algebras.
\end{abstract}
\maketitle

\section{Introduction}
A quandle is an algebraic system with a single binary operation that satisfies axioms encoding the three Reidemeister moves of planar diagrams of links in the 3-space. These objects show appearance in diverse areas of mathematics, namely, knot theory \cite{Joyce, Matveev}, group theory, set-theoretic solutions to the Yang-Baxter equations and Yetter-Drinfeld Modules \cite{Eisermann}, Riemannian symmetric spaces \cite{Loos} and Hopf algebras \cite{Andruskiewitsch}, to name a few. Though already studied under different guises in the literature, study of these objects gained momentum after the fundamental works of Joyce \cite{Joyce} and Matveev \cite{Matveev}, who showed that link quandles are complete invariants of non-split links up to orientation of the ambient 3-space. Although link quandles are strong invariants, it is difficult to check whether two quandles are isomorphic. This motivated search for newer properties and invariants of quandles themselves. We refer the reader to the articles \cite{Carter, Kamada, Nelson} for more on the historical development of the subject.
\par

In recent years, quandles and their weaker analogues called racks have received a great deal of attention. A (co)homology theory for quandles and racks has been developed in \cite{Carter2, Nosaka}, which has led to stronger invariants of links. In fact, a recent work \cite{Szymik} shows that quandle cohomology is a Quillen cohomology which is the cohomology group of a functor from the category of models (or algebras) to that of complexes. Automorphisms of quandles, which reveal a lot about their internal structures, have been investigated in much detail in a series of papers \cite{BDS, BarTimSin, Elhamdadi2012}. Fusing ideas from combinatorial group theory into quandles, recent works \cite{BSS1, BSS2} show that free quandles and link quandles are residually finite. 
\par

In an attempt to linearise the study of quandles, a theory of quandle rings analogous to the classical theory of group rings was proposed by the authors in \cite{BPS}, where several interconnections between quandles and their associated quandle rings were investigated, and an analogue of the group rings isomorphism problem for quandle rings was proposed. The work was carried forward in a recent paper \cite{EFT} of Elhamdadi et al., where examples of non isomorphic finite quandles with isomorphic quandle rings have been given. At the same time, they proved that if two finite quandles admit 2-transitive actions of their inner automorphism groups and have isomorphic quandle rings, then the quandles have the same number of orbits of each cardinality.
\par

The purpose of this paper is to develop the theory of quandle rings further. It may be mentioned that, at this point, our approach and motivation is purely algebraic. However, we do propose a natural problem concerning knots and links (Problem \ref{link-quandles-orderable}). Following \cite{BPS}, given a quandle (resp. rack) $Q$ and an integral domain $R$, the quandle (resp. rack) ring $R[Q]$ of $Q$ with coefficients in $R$ is defined as the set of all formal finite $R$-linear combinations of elements of $Q$ with usual operations (see Section \ref{sec-prelim}). We investigate zero-divisors in quandle rings by introducing orderability in quandles and show that many interesting quandles arising from orderable groups are left or right orderable. Investigation of unit groups of group rings is a major research problem in the subject. An analogue of this problem for quandles is the investigation of maximal quandles in quandle rings. We show that  the set $\mq(R[Q])$ of all non-zero maximal quandles in $R[Q]$ contains, in general, more than one element. Since each element of a quandle is an idempotent in its quandle ring, the first step towards a solution of the problem is to describe the set $I(R[Q])$ of all non-zero idempotents in $R[Q]$. The problem of determining $\mq(R[Q])$ also connects with the description of the group $\Aut(R[Q])$ of ring automorphisms of $R[Q]$ that are $R$-linear. Clearly every automorphism of $Q$ induces an automorphism of  $R[Q]$. In fact, any automorphism $\phi \in \Aut(R[Q])$ is defined by its action on $Q$ and its image $\phi(Q)$ lies in some maximal quandle from $\mq(R[Q])$. We compute idempotents, maximal quandles and $R$-algebra automorphisms of quandle rings of small order quandles including all quandles of order 3.
\par

The paper is organised as follows. In Section \ref{sec-prelim}, we recall some basic definitions and examples from the theory of  quandles and quandle rings. In Section \ref{orderability-zero-divisors}, we introduce unique product quandles and show that their quandle rings have  no zero-divisors over integral domains. We define orderability of quandles to give explicit examples of such quandles, and show that a semi-latin quandle which is right or left orderable is necessarily a unique product quandle. Our results also answer a question from \cite[Question 4.3]{EFT} about existence of quandles whose quandle rings do not have zero-divisors, and suggest an analogue of  Kaplansky's zero-divisor conjecture for quandle rings. In Section \ref{sec-idempotents}, we compute idempotents in quandle rings $R[\T]$, $\mathbb{Z}[\R_3]$, $\mathbb{Z}[\R_4]$ and $\mathbb{Z}[\Cs(4)]$, where $\T$ is any trivial quandle, $\R_n$ is a dihedral quandle and $\Cs(4)$ is the 3-element singular cyclic quandle of Joyce \cite{Joyce-Thesis}. In Section \ref{sec-max-quandles}, the computation of idempotents is then used to determine the set of maximal quandles in these quandle rings. In Section \ref{sec-auto}, we determine automorphism groups of these quandle rings. More precisely, we prove that $\Aut(\mathbb{Z}[\T_2]) \cong \mathbb{Z} \rtimes \mathbb{Z}_2$, $\Aut(\mathbb{Z}[\R_3]) \cong \Aut(\R_3)$, $\Aut (\mathbb{Z}[\R_4]) \cong ( \mathbb{Z}_2 \times \mathbb{Z}_2) \rtimes \mathbb{Z}_2$ and $\Aut(\mathbb{Z}[\Cs(4)])\cong \mathbb{Z}_2$, where $\T_2$ is the 2-element trivial quandle. In Section, \ref{commutator-width}, we introduce commutator width of quandle rings and give a bound for the commutator width of finite non-commutative quandles admitting a 2-transitive action by their automorphism groups. We also compute the precise commutator width of quandle rings $R[\T]$, $R[\R_3]$, $R[\R_4]$ and $R[\Cs(4)]$. We conclude with Section \ref{relations-other-algebras} where we comment on relation of quandle algebras with other well-known non-associative algebras like alternative  algebras, Jordan algebras and Lie algebras.
\medskip

\section{Preliminaries on quandle rings}\label{sec-prelim}
A {\it quandle} is a non-empty set $Q$ with a binary operation $(x,y) \mapsto x * y$ satisfying the following axioms:
\begin{enumerate}
\item[(Q1)] $x*x=x$ for all $x \in Q$,
\item[(Q2)] For any $x,y \in Q$ there exists a unique $z \in Q$ such that $x=z*y$,
\item[(Q3)] $(x*y)*z=(x*z) * (y*z)$ for all $x,y,z \in Q$.
\end{enumerate}

An algebraic system satisfying only (Q2) and (Q3)  is called a {\it rack}. Many interesting examples of quandles come from groups showing deep connection with group theory.
\par

\begin{itemize}
\item If $G$ is a group, then the binary operation $a*b= b^{-1} a b$ turns $G$ into the quandle $\Conj(G)$ called the {\it conjugation quandle} of $G$.
\item A group $G$ with the binary operation $a*b= b a^{-1} b$ turns the set $G$ into the quandle $\Core(G)$ called the {\it core quandle} of $G$. In particular, if $G= \mathbb{Z}_n$, the cyclic group of order $n$, then it is called the {\it dihedral quandle} and denoted by $\R_n$.
\item Let $G$ be a group and $\phi \in \Aut(G)$. Then the set $G$ with binary operation $a * b = \phi(ab^{-1})b$ forms a quandle $\Alex(G,\phi)$ referred as the  {\it generalized Alexander quandle} of $G$ with respect to $\phi$.
\end{itemize}
\medskip

A quandle  $Q$ is called {\it trivial} if $x*y=x$ for all $x, y \in Q$.  Unlike groups, a trivial quandle can have arbitrary number of elements. We denote the $n$-element trivial quandle by $\T_n$ and an arbitrary trivial quandle by $\T$.
\medskip

Notice that the axioms (Q2) and (Q3) are equivalent to the map $S_x: Q \to Q$ given by $$S_x(y)=y*x$$ being an automorphism of $Q$ for each $x \in Q$. These automorphisms are called {\it inner automorphisms}, and the group generated by all such automorphisms is denoted by $\Inn(X)$. A quandle is said to be {\it connected}  if it admits a transitive action by its group of inner automorphisms. For example, dihedral quandles of odd order are connected, whereas that of even order are disconnected. A quandle $X$ is called {\it involutary} if $S_x^2 = \id_Q$ for each $x \in Q$. For example, all core quandles are involutary.  A quandle (resp. rack) $Q$ is called {\it commutative} if $x * y = y * x$ for all $x, y \in Q$. The dihedral quandle $\R_3$ is commutative and no trivial quandle with more than one element is commutative. 
\par

A quandle $Q$ is called {\it latin} if left multiplication by each element of $Q$ is a bijection of $Q$, that is, the map $L_x : Q \to Q$ defined by $$L_x(y) = x *y$$ is a bijection for each $x \in Q$. For example, $\R_3$ is latin but no trivial quandle with more than one element is latin. We say that $Q$ is  {\it semi-latin} if left multiplication by each element of $Q$ is an injection of $Q$. Obviously, every latin quandle is semi-latin. The converse is not true in general; for example, the quandle $\Core(\mathbb{Z})$ is semi-latin but not latin. In fact, a direct check shows that if $G$ is an abelian group, then $\Core(G)$ is semi-latin if and only $G$ has no 2-torsion. Similarly, one can see that for an arbitrary group $G$ and an automorphism $\phi \in \Aut(G)$, the quandle $\Alex(G, \phi)$ is semi-latin if and only if $\phi$ is fixed-point free.
\medskip

Next we recall some definitions and results from \cite{BPS}. Throughout this paper, unless specified otherwise, $R$ will be an integral domain, that is, an associative and commutative ring with unity and without zero-divisors. From now onwards, except the situation where there are more than one binary operations on a set, we denote the multiplication in a quandle (resp. rack) by $(x,y) \mapsto xy$.
\medskip

Let $Q$ be a quandle and  $R[Q]$ the set of all formal finite $R$-linear combinations of elements of $Q$, that is,
$$R[Q]:=\Big\{ \sum_i\alpha_i x_i~|~\alpha_i \in R,~ x_i \in Q \Big\}.$$
Then $R[Q]$ is an additive abelian group with coefficient-wise addition. Define multiplication in $R[Q]$ by setting $$\big(\sum_i\alpha_i x_i\big) \big(\sum_j\beta_j x_j\big):=\sum_{i,j}\alpha_i\beta_j (x_i x_j).$$
Clearly, the multiplication is distributive with respect to addition from both left and right, and $R[Q]$ forms a ring (in fact, an $R$-algebra), which we call the {\it quandle ring} (or {\it quandle algebra}) of $Q$ with coefficients in the ring $R$. Since $Q$ is non-associative, unless it is a trivial quandle, it follows that $R[Q]$ is a non-associative ring in general. If $Q$ is a rack, then its {\it rack ring} (or {\it rack algebra}) $R[Q]$ is defined analogously.  
\par

Define the augmentation map
$$\varepsilon: R[Q] \to R$$
 by setting $$\varepsilon \big(\sum_i\alpha_i x_i\big)= \sum_i\alpha_i .$$
Clearly, $\varepsilon$ is a surjective ring homomorphism, and $\Delta_R(Q):= \ker(\varepsilon)$ is a two-sided ideal of $R[Q]$, called the {\it augmentation ideal} of $R[Q]$. It is easy to see that $\{x-y~|~x, y \in Q \}$ is  a generating set for $\Delta_R(Q)$ as an $R$-module. Further, if $x_0 \in Q$ is a fixed element, then the set $\big\{x-x_0~|~x \in Q \setminus \{ x_0\} \big\}$ is a basis for $\Delta_R(Q)$ as an $R$-module. For convenience, we denote $\Delta_\mathbb{Z}(Q)$ by $\Delta(Q)$. 
\par

Since $R[Q]$ is a ring without unity, it is desirable to embed it into a ring with unity. The ring
$$
R^\circ[Q]:=R[Q] \oplus Re,
$$
where $e$ is a symbol (not in $Q$) satisfying $e\big(\sum_i\alpha_i x_i\big)= \sum_i\alpha_i x_i= \big(\sum_i\alpha_i x_i\big)e$, is called the {\it extended quandle ring} of $Q$. For convenience, we denote the unity $1e$ of $R^\circ[Q]$ by $e$.  We can extend the augmentation map to $\varepsilon: R^\circ[Q] \to R$ and define the {\it extended augmentation ideal} as
$$\Delta_{R^\circ}(Q):= \ker(\varepsilon: R^\circ[Q] \to R).$$
As before, it is easy to see that the set $\{x-e~|~ x \in Q  \}$ is a basis for $\Delta_{R^\circ}(Q)$ as an $R$-module.
\par

We conclude the section by recalling a result that characterises trivial quandles in terms of their augmentation ideals \cite[Theorem 3.5]{BPS}.
\par

\begin{theorem} \label{deltasqzero}
A quandle $Q$ is trivial if and only if $\Delta_R^2(Q)=\{0\}$.
\end{theorem}
\medskip

\section{zero-divisors in quandle rings}\label{orderability-zero-divisors}
Recall that a non-zero element $u$ of a ring is called a {\it zero-divisor} if there exists a non-zero element $v$ such that either $uv=0$ or $vu=0$. Every non-zero nilpotent element of an associative ring is a zero-divisor. Determining whether group rings of torsion-free groups over fields have zero-divisors is a classical and still open problem in the theory of group rings. In this section, we investigate the analogous problem for quandle rings.
\par
Let $R$ be an integral domain. It is easy to see that if $\T$ is a trivial quandle with more than one element, then $R[\T]$ contains zero-divisors. If $G$ is a group with an element $g$  of finite order, say $n>1$, then the element $$\hat{g}:= 1+ g+ \cdots+ g^{n-1}$$ of the integral group ring $\mathbb{Z}[G]$ satisfies
$\hat{g}(1-g)=0$, and hence $\mathbb{Z}[G]$ has a  zero-divisor. By analogy, it was proved in \cite[Proposition 4.1]{EFT} that, if a quandle $Q$ has a finite orbit (under the action of $\Inn(Q)$) with more than one element, then $R[Q]$ has zero-divisors.
\par

We first formulate some sufficient conditions under which a quandle ring contains zero-divisors. We say that a quandle $Q$ containing more than one element is {\it inert} if there is a finite subset $A = \{ a_1, a_2, \ldots, a_n \}$ of $Q$ and two distinct elements $x, y \in Q$ such that $A x = A y$, where $A z = \{ a_1 z, a_2 z, \ldots, a_n z \}$.

\begin{proposition}\label{two different elements}
The following hold:
\begin{enumerate}
\item Any extended quandle ring $R^{\circ}[Q]$ contains zero-divisors.
\item  If $Q$ is a quandle containing a trivial subquandle with more than one element, then $R[Q]$ contains zero-divisors.
\item If $Q$ is an inert quandle, then the quandle ring $R[Q]$ contains zero-divisors. In particular, if $Q$ contains a finite subquandle with more than one element, then the quandle ring $R[Q]$ contains zero-divisors.
\item If $Q$ is not semi-latin, then the quandle ring $R[Q]$ contains zero-divisors.
\end{enumerate}
\end{proposition}

\begin{proof}
If $e$ is the unit in $R^{\circ}[Q]$ and $x \in Q$, then
$$
x (e - x) = x - x^2 = x - x = 0,
$$
Thus, $x$ and $e-x$ are zero-divisors, which proves (1).
\par
For (2), let $T = \{ x_1, x_2 \}$ be a trivial subquandle in $Q$. Taking $u = x_1 - x_2 \in R[Q]$ gives
$$
u^2 = (x_1 - x_2) (x_1 - x_2) = 0.
$$
\par

For (3), let $x$ and $y$ be two distinct elements in $Q$ and  $A = \{ a_1, a_2, \ldots, a_m \}$ such that $A x = A y$. Then
$$
(a_1 + a_2 + \cdots + a_k) (x - y) = 0.
$$
If $Q$ contains a finite subquandle $A$, then we can take $x$ and $y$ to be two distinct elements of $A$.
\par
For (4), suppose that for some $x\in Q$ there exist distinct $y, z \in Q$ such that $L_x(y) = L_x(z)$. Then we have
$$
x (y-z) = L_x(y) - L_x(z) = 0.
$$
\end{proof}

If  $R$ is an integral domain, then it is obvious that the quandle ring $R[T_1]$ of the one element quandle $T_1$ does not have zero-divisors.  The following question was raised in \cite[Question 4.3]{EFT}. 

\begin{question}\label{Elhamdadi-question}
Are there other quandles $Q$ for which $R[Q]$ does not have zero-divisors?
\end{question}

We introduce a class of quandles whose quandle rings do not have zero-divisors.  As in case of groups (see, for example, \cite[Chapter 13]{Passman}), a quandle $Q$ is said to be a {\it up-quandle} (unique product quandle) if given any two non-empty finite subsets $A$ and $B$ of $Q$, there is at least one element $x \in Q$ that has a unique representation of the form $x = a b$ for some $a \in A$ and $b\in B$. A quandle $Q$ is said to be a {\it tup-quandle} (two unique product quandle) if given any two non-empty finite subsets $A$ and $B$ of $Q$ with $|A| + |B| > 2$, there exists  at least two distinct elements $x, y \in Q$ that have unique representations of the form $x = a b$ and $y = c d$, where $a, c \in A$ and $b, d \in B$. It is clear that every t.u.p-quandle is a up-quandle.
\par

The following observation is an analogue of the corresponding result for groups \cite[Chapter 13, Lemma 1.9]{Passman}.

\begin{proposition} \label{up-zero}
If $Q$ is a up-quandle, then $R[Q]$ has no zero-divisors.
\end{proposition}

\begin{proof}
Let $u$ and $v$ be non-zero elements of $R[Q]$ and write
$$
u = \sum_{i=1}^n \alpha_i x_i,~~~v = \sum_{j=1}^m \beta_j y_j,
$$
where $\alpha_i, \beta_j$ are non-zero elements of $R$ and $A = \{ x_i \}$ and $B = \{ y_j \}$ are non-empty subsets of $Q$. Then
$$
u v = \sum_{i, j} \alpha_i \beta_j x_i  y_j,
$$
where each $\alpha_i \beta_j  \not= 0$ since $R$ has no zero-divisors. Since $Q$ is a up-quandle, there exists a uniquely represented element in the product $A B$, say $z = x_1 y_1$. It then follows that non-zero summand $\alpha_1 \beta_1 x_1  y_1$ cannot be cancelled by any other term in the product $uv$. Thus, $u v \not= 0$ and $R[Q]$ has no zero-divisors.
\end{proof}

We now introduce orderable quandles to give explicit examples of up-quandles. Following the notion of orderability of groups \cite[Chapter 13]{Passman}, we say that a quandle $Q$ is \textit{right orderable} if the elements of $Q$ are linearly ordered with respect to a  relation $<$ such that  $x < y$ implies $xz < yz$ for all $x, y, z \in Q$.  Similarly, we say that $Q$ is \textit{left orderable} if the elements of $Q$ are linearly ordered with respect to a  relation $<$ such that $x < y$ implies $zx < zy$ for all $x, y, z \in Q$. A  quandle is said to be \textit{bi-orderable} (or simply \textit{orderable}) if it is both left and right orderable. Note that the definitions make sense for racks as well.
\par

A right orderable group $G$ must also be left orderable and vice-versa, but not under the same ordering. Indeed, if $<$ is a right ordering for  $G$, then it is easy to see that $<'$ defined by $x <' y$ if and only if $y^{-1} < x^{-1}$ yields a left ordering (see \cite[Chapter 13]{Passman}). However, the case of quandles is not the same. For example, a trivial quandle can be right orderable but not left orderable. Indeed, if $\T = \{ x_1, x_2, \ldots \}$ is a trivial quandle, then it is clear that the linear order $x_1 < x_2 < \cdots$ is preserved under multiplication on the right, but is not preserved under multiplication on the left.
\par

The following result gives examples of some left and right orderable quandles arising from groups.

\begin{proposition}\label{core-orderable}
The following hold for an orderable group $G$:
\begin{enumerate}
\item  $\Conj(G)$ is a right orderable quandle.
\item $\Core(G)$ is a left orderable quandle. 
\item If $\phi \in \Aut(G)$ is an order reversing automorphism, then $\Alex(G, \phi)$ is a left orderable quandle.
\end{enumerate}
\end{proposition}

\begin{proof}
Let $G$ be an ordered group with order $<$ and $x, y, z \in G$ such that $x < y$. Then $$x*z=z^{-1}xz < z^{-1}yz=y*z$$ implies that $\Conj(G)$ is a right orderable quandle, and
$$z*x= xz^{-1}x <  yz^{-1}x <  yz^{-1}y= z*y$$ implies that $\Core(G)$ is a left orderable quandle. This proves (1) and (2).
\par
For (3), ordering of $G$ and $\phi$ being order reversing implies that $\phi(x)^{-1} < \phi(y)^{-1}$. This gives
$$z*x= \phi(zx^{-1})x=  \phi(z)\phi(x^{-1})x < \phi(z)\phi(x^{-1})y < \phi(z)\phi(y^{-1})y=z*y,$$
which proves that $\Alex(G, \phi)$ is left orderable.
\end{proof}

We recall the construction of the free quandle on a given set (\cite[p.351]{Fenn-Rourke}, \cite{Kamada2012}). Let $S$ be a set and $F(S)$ the free group on $S$. Define
 $$FR(S):= \big\lbrace a^{w} \mid a \in S, w \in F(S) \rbrace $$ with the operation given as $$a^{w} \ast b^{u}:= a^{wu^{-1}bu}.$$
A direct check shows that $FR(S)$ is a free rack on $S$.  The free quandle $FQ(S)$ on $S$ is then defined as a quotient of $FR(S)$ modulo the equivalence relation generated by $$a^{w}= a^{aw}$$ for $a \in S$ and $w \in F(S)$. It is easy to see that $FQ(S)$ is the desired free quandle satisfying the universal property.
If $|S|=n$, we denote $FQ(S)$ by $FQ_n$. With this definition, we have the following result.

\begin{theorem}\label{ordrability-free-quandle}
Free quandles are right orderable and semi-latin.
\end{theorem}

\begin{proof}
It is known that the free group $F_n = \langle x_1, x_2, \ldots, x_n \rangle$ is orderable  \cite{Vinogradov, Deroin}. Consequently, $\Conj(F_n)$ is right orderable by  Proposition \ref{core-orderable}, and hence the free quandle $FQ_n$ is right orderable being a subquandle of $\Conj(F_n)$.
\par
Let us prove that $FQ_n$ is semi-latin. If $n=1$, then $FQ_n$ is the one-element trivial quandle and assertion is evident. Suppose that $n > 1$ and there are elements $x, y, z \in FQ_n$ such that $x \not= y$ and $z * x = z * y$. Using the interpretation of elements of $FQ_n$ as elements in $F_n$, we can assume that
$$
z = x_i^{z_0},~~x = x_j^{x_0},~~y = x_k^{y_0},~~z_0, x_0, y_0 \in F_n,
$$
where $a^b := b^{-1} a b$. The identity $z * x = z * y$ gives the equality
$$
x_i^{z_0 x_0^{-1} x_j x_0} = x_i^{z_0 y_0^{-1} x_k y_0}
$$
in the free group $F_n$, which is equivalent to
$$
x_i^{z_0 x_0^{-1} x_j x_0 y_0^{-1} x_k^{-1}  y_0 z_0^{-1}} = x_i.
$$
But it is possible in $F_n$ if and only if 
\begin{equation} \label{e1}
z_0 x_0^{-1} x_j x_0 y_0^{-1} x_k^{-1}  y_0 z_0^{-1} = x_i^{\alpha}
\end{equation}
for some integer $\alpha$. Take the quotient of $F_n$ by its commutator subgroup, the previous equality gives
$$
\overline{x_j} \cdot \overline{x_k}^{-1} = \overline{x_i}^{\alpha},
$$
where $\overline{x_j},  \overline{x_k},  \overline{x_i}$ are the generators of the free abelian group $F_n/F_n'$. Hence, $j = k$ and $\alpha = 0$. Thus, (\ref{e1}) has the form
$$
z_0 x_0^{-1} x_j x_0 y_0^{-1} x_j^{-1}  y_0 z_0^{-1} = 1.
$$
Conjugating both sides by $z_0$ gives
$$
x_0^{-1} x_j x_0 y_0^{-1} x_j^{-1}  y_0 = 1,
$$
or
$$
[x_j^{-1}, y_0 x_0^{-1}] = 1,
$$
where $[a, b] = a^{-1} b^{-1} ab$. This equality holds if and only if $y_0 x_0^{-1} = x_j^{\beta}$ for some integer $\beta$, i.e. $y_0 = x_j^{\beta} x_0$.
Hence, the elements $z, x, y$ have the form
$$
z = x_i^{z_0},~~ x = x_j^{x_0},~~ y= x_j^{x_0},
$$
i.e. $x = y$. This contradiction proves that $FQ_n$ is a semi-latin quandle.
\end{proof}

It is interesting to have an answer to the following question.

\begin{question}
Does there exists an infinite non-commutative bi-orderable quandle?
\end{question}

It is easy to see that a right or left orderable group must be infinite. But, this is not true for quandles since any finite trivial quandle is right orderable. However, the following properties hold.

\begin{proposition}\label{orderable-implies-latin}
Let $Q$ be a quandle. Then the following hold:
\begin{enumerate}
\item If $Q$ is  right orderable, then the $\langle S_y \rangle $-orbit of  $x$ is infinite for all $x, y \in Q$ with $S_y(x) \ne x$.
\item If $Q$ is  left orderable, then it is semi-latin and the set $\{L_y^n(x)\}_{n=0, 1, \ldots}$ is infinite for $x \ne y \in Q$, where
$$
L_y^0(x) = x,~~L_y^{i+1}(x) = y * (L_y^{i}(x)),~~i = 0, 1, \ldots.
$$
\end{enumerate}
\end{proposition}

\begin{proof}
If $x<S_y(x)$ and the $\langle S_y \rangle$-orbit of $x$ is finite, then right orderability of $Q$ implies that  $$x <  S_y(x) <  S_y^2(x) <  \cdots <  S_y^n(x)=x$$ for some integer $n$, which is a contradiction. Similarly, the assertion follows if $S_y(x) < x$.
\par

Suppose that there are elements $x, y, z \in Q$ with $y\neq z$, say $y < z$, such that $x * y = x * z$. This is a contradiction to left orderability of $Q$, and hence $Q$ must be semi-latin. Further, if $x \ne y$ are two elements of $Q$ such that $x = L_y(x)$, then $x * x = y * x$, which contradicts the second quandle axiom. Hence, $x<L_y(x)$ or $L_y(x) < x$. Suppose that $x<L_y(x)$. Since $Q$ is left orderable, we have 
$$
x < L_y(x) < L_y^2(x) < \cdots < L_y^n(x)< \cdots,
$$
and hence $\{L_y^n(x)\}_{n=0, 1, \ldots}$ is infinite. The case $L_y(x) < x$ is similar.
\end{proof}

\begin{corollary}
If $G$ is a non-trivial group, then $\Conj(G)$ is not left orderable and $\Core(G)$ is not right orderable.
\end{corollary}

\begin{proof}
Since $\Conj(G)$ is not semi-latin and $\Core(G)$ is involutary, the assertions follow from Proposition \ref{orderable-implies-latin}.
\end{proof}

If $\phi \in \Aut(G)$ is an involution, then $\Alex(G, \phi)$ is involutary and we obtain

\begin{corollary}
If $G$ is a non-trivial group and $\phi \in \Aut(G)$ an involution, then the quandle $\Alex(G, \phi)$ is not right orderable.
\end{corollary}

It is known that the  group ring of a right orderable group has no zero-divisors \cite[Chapter 13]{Passman}. On the other hand, a trivial quandle  with more than one element is right orderable and its quandle ring always has zero-divisors. However, for semi-latin quandles we have the following result, which is a quandle analogue of \cite[Chapter 13, Lemma 1.7]{Passman} and also answers Question \ref{Elhamdadi-question}.

\begin{proposition} \label{orderable-tup}
Let $Q$ be a semi-latin quandle. If $Q$ is right or left orderable, then $Q$ is a t.u.p-quandle. In fact, if $A$ and $B$ are non-empty  finite subsets of $Q$, then there exist $b', b'' \in B$ such that the product $a_{max} b'$ and $a_{min} b''$ are uniquely represented in $A B$, where $a_{max}$ denotes the largest element in $A$ and $a_{min}$ the smallest.
\end{proposition}

\begin{proof} Suppose that $Q$ is a semi-latin and right orderable quandle. Let
$$
A = \{a_{min} = a_1 < a_2 < \cdots < a_n = a_{max} \},~~n \geq 2,
$$ 
and 
$$
B = \{ b_1 < b_2 < \cdots < b_m \}
$$ 
be two finite subsets of $Q$. We write the elements of the product $A B$ in the tabular form

$$
\begin{array}{ccccccc}
a_1 b_1 & < & a_2 b_1 & < & \cdots & < & a_n b_1, \\
a_1 b_2 & < & a_2 b_2 & < & \cdots & < & a_n b_2, \\
\vdots & \vdots & \vdots & \vdots & \vdots & \vdots & \vdots \\
a_1 b_m & < & a_2 b_m & < & \cdots & < & a_n b_m, \\
\end{array}
$$
where the inequalities in the rows follow from the right ordering of $Q$. Since $Q$ is semi-latin, it follows that all  the entries in each column are distinct.
\par
Let $b_i \in B$ be the element  such that $a_1 b_i$ is the minimal element in the first column.  Let us prove that we can take $b'' = b_i$, i.e.
$a_{min} b'' = a_1 b_i$ is uniquely represented in $A B$. It suffices to prove that $a_1 b_i < a_k b_l$ for any pair $(k, l) \not= (1, i)$. If $k=1$, then 
the inequality $a_1 b_i < a_1 b_l$, $l \not= i$, follows from the choice of $b_i$. If $k >1$, then $a_1 b_i \leq a_1 b_l$ and the inequalities in the $l$-th row  imply that $a_1 b_i < a_k b_l$.
\par
Let $b_j \in B$ be the element such that $a_n b_j$ is the maximal element in the last column. We prove that one can take $b' = b_j$, that is, $a_k b_l < a_n b_j$ for each $(k, l) \not= (n, j)$. If $k = n$, then the inequality follows from the choice of $b_j$. If $k < n$, then  inequalities in the $l$-th row gives
$$
a_k b_l < a_n b_l \leq a_n b_j.
$$
Hence, the product $a_{max} b' = a_n b_j$ is  uniquely represented in $A B$. The case when $Q$ is left orderable is similar.
\end{proof}

Propositions \ref{up-zero} and \ref{orderable-tup}  together yield the following result.

\begin{theorem} \label{orderable-latin}
Let  $Q$ be a semi-latin quandle. If $Q$ is right or left orderable, then $R[Q]$ has no zero-divisors.
\end{theorem}

Theorem \ref{ordrability-free-quandle} and Theorem \ref{orderable-latin}  leads to the following result.

\begin{corollary}\label{free-quandle-zero-divisors}
Quandle rings of free quandles have no  zero-divisors.
\end{corollary}

As a consequence of Proposition \ref{core-orderable}, Proposition \ref{orderable-implies-latin}(2) and Theorem \ref{orderable-latin}, we have the following results.

\begin{corollary}\label{core-zero-divisors}
If $G$ is an orderable group, then the quandle ring $R[\Core(G)]$ has no  zero-divisors.
\end{corollary}

\begin{corollary}\label{alex-zero-divisors}
If $G$ is an orderable group and $\phi \in \Aut(G)$ an order reversing automorphism, then the quandle ring $R[\Alex(G, \phi)]$ has no  zero-divisors.
\end{corollary}

\par
Proposition \ref{two different elements} and Theorem \ref{orderable-latin} suggest the following analogue of Kaplansky's zero-divisor conjecture for quandles.

\begin{conjecture}\label{kaplansky-analogue}
Let $R$ be an integral domain and $Q$ a non-inert semi-latin quandle.  Then the quandle ring $R[Q]$ has no zero-divisors.
\end{conjecture}

It is known that all link groups are left orderable \cite{Boyer-Rolfsen-Wiest}, whereas not all knot groups are bi-orderable \cite{Perron-Rolfsen}. For example,  the group of the figure-eight knot is bi-orderable and the group of a non-trivial cable of an arbitrary knot is not bi-orderable. Since knot quandles are deeply related to knot groups,  the following problem seems interesting.

\begin{problem}\label{link-quandles-orderable}
Determine whether link quandles are left or right orderable.
\end{problem}

\medskip

\section{Idempotents in quandle rings}\label{sec-idempotents}
The computation of idempotents is an important problem in ring theory. The study of idempotents in quandle rings is also motivated by the search for new  quandles  contained in quandle rings. To compute the set $I(R[Q])$ of non-zero idempotents in a given quandle ring $R[Q]$, we begin with some general observations. First notice that each quandle element is, by definition,  an idempotent in its quandle ring and we  refer to them as {\it trivial idempotents}. 

It is well-known that integral group rings do not have non-trivial idempotents (see \cite[p. 123]{Kaplansky} or \cite[p. 38]{Passman}). In sharp contrast, in extended quandle rings, the identity element is trivially an idempotent, and therefore the elements $e-x$, with $x\in Q$, too are idempotents. 

Since the augmentation map $\varepsilon : R[Q] \to R$  is a ring homomorphism, it maps idempotent in $R[Q]$ to idempotents in $R$. Since $R$ is an integral domain, $\varepsilon(z) = 0$ or $\varepsilon(z) = 1$ for each idempotent $z$ of $R[Q]$. In the first case $z \in \Delta_R (Q)$, and in the second case $z = x + \delta$ for some $x \in X$ and $\delta \in \Delta_R (Q)$. 

\begin{proposition} \label{trivial-quandle-idempotents}
If $\T$ is a trivial quandle, then $ I(R[\T]) =  x_0 + \Delta_R (\T)$, where  $x_0 \in \T$ is a fixed element.
\end{proposition}

\begin{proof}
Since $\T$ is trivial, by Theorem \ref{deltasqzero}, $\Delta_R^2 (\T) = 0$. It follows that non-zero idempotents do not lie in $\Delta_R (\T)$. Hence, a non-zero idempotent has the form
 $z = x_0 + \delta$, where $\delta \in \Delta_R (\T)$ and $x_0 \in \T$ some fixed element. Indeed, 
$$
z^2 = x_0^2 + x_0 \delta + \delta x_0 + \delta^2= x_0+ \delta=z
$$
since $x_0^2 = x_0$, $x_0 \delta = \delta^2 = 0$ and $\delta x_0 = \delta$.
\end{proof}
\medskip

Observe that, if a quandle $Q = Q_1 \sqcup Q_2$ is a disjoint union of two subquandles, then
\begin{equation}\label{idempotents-union}
I(R[Q]) \supseteq  I(R[Q_1]) \cup I(R[Q_2]) .
\end{equation}
The inclusion is, in general, not an equality, as we see from the following result.

\begin{proposition}\label{joyce-quandle}
Let $\Cs(4)$ be the 3-element singular cyclic quandle given by
$$
\Cs(4) = \big\langle x, y, z~|~ x^2 = x,~y^2 = y,~z^2 = z,~x y = x,~x z = y,~y x = y,~y z = x,~z x = z,~z y = z \big\rangle.
$$
Then $I(\mathbb{Z}[\Cs(4)]) = \big\{ (1 - \beta) x +  \beta y,~ \alpha x + \alpha y +  (1 - 2 \alpha) z~|~\beta, \alpha \in \mathbb{Z} \big\}$.
\end{proposition}

\begin{proof}
The quandle $\Cs(4)$ is a disjoint union of two trivial subquandles, i.e. $\Cs(4)= \{ x, y \} \sqcup \{z \}$. If $w = \alpha x + \beta  y + \gamma z \in \mathbb{Z}[\Cs(4)]$, then
$$
w^2 = (\alpha^2 + \alpha \beta + \beta \gamma) x + (\alpha \beta + \beta^2 + \alpha \gamma) y + (\alpha \gamma + \beta \gamma + \gamma^2 )z.
$$
Thus, $w$ is an idempotent if and only if the system of equations
\begin{eqnarray*}
\alpha &=& \alpha^2 + \alpha \beta + \beta \gamma,\\
\beta &=& \beta^2 + \alpha \beta +  \alpha \gamma,\\
\gamma &=& \gamma^2 + \alpha \gamma + \beta \gamma,
\end{eqnarray*}
is simultaneously  solvable over integers. Suppose that $\gamma = 0$. Then we have the equations
$$\alpha = \alpha^2 + \alpha \beta,~\beta = \beta^2 + \alpha \beta.$$
If $\alpha = 0$, then $\beta = 0$ or $\beta = 1$. In the first case $w=0$, and in the second case $w = y$. If $\alpha \not= 0$, then $\alpha = 1 - \beta$ and we have idempotents
$$
w = (1 - \beta) x +  \beta y,~~\beta \in \mathbb{Z}.
$$
These are idempotents of the quandle ring $\mathbb{Z}[\{ x, y \}]$.
\par

Suppose further that $\gamma \not= 0$. Then the third equation of the system gives $\gamma = 1 - \alpha - \beta$. Substituting this expression in the first and the second equations gives
$$
(\alpha - \beta) (\alpha + \beta) = (\alpha - \beta).
$$
If $\alpha - \beta \not= 0$, then we have the same idempotent as in the previous case. If $\alpha - \beta = 0$, then we have idempotents
$$
w = \alpha x + \alpha y +  (1 - 2 \alpha) z,~~\alpha \in \mathbb{Z}.
$$
Thus, we have
$$
I(\mathbb{Z}[\Cs(4)]) = \big\{ (1 - \beta) x +  \beta y, \alpha x + \alpha y +  (1 - 2 \alpha) z~|~\beta, \alpha \in \mathbb{Z} \big\}.
$$
\end{proof}

The preceding example shows that
$$
I(\mathbb{Z}[\Cs(4)]) \not= I(\mathbb{Z}[\{ x, y \}]) \cup I(\mathbb{Z}[\{ z \}]),
$$
and hence the inclusion in  \eqref{idempotents-union} is, in general, strict.
 \medskip

Let  $\R_n = \{ a_0, a_1, \ldots, a_{n-1} \}$ be the dihedral quandle of order $n$, where $a_i*a_j=a_{2j-i ~(mod~n)}$. We examine idempotents in $\mathbb{Z}[\R_n]$ for $1 \le n\le 4$. Note that $\R_1$ and $\R_2$ are trivial quandles.

\begin{proposition}\label{idempotents-r3}
$ I(\mathbb{Z}[\R_3]) = \{ a_0, a_1, a_2 \}. $
\end{proposition}

\begin{proof}
Let $z = \alpha_0 a_0 + \alpha_1 a_1 + \alpha_2 a_2 \in \mathbb{Z}[\R_3]$ be an idempotent. Then $\varepsilon(z) = 0$ or $\varepsilon(z) = 1$.
\medskip

 Case 1. $\varepsilon(z) = 0$, i.e. $\alpha_0 = - \alpha_1 - \alpha_2$. Then $z = \alpha_1 e_1 + \alpha_2 e_2$, where $e_i = a_i - a_0$. The elements $e_1$ and $e_2$ generate $\Delta (\R_3)$ and have the following multiplication table.
\begin{center}
\begin{tabular}{|c||c|c|}
  \hline
 $\cdot$ & $e_1$ & $e_2$  \\
   \hline
     \hline
$e_1$ & $e_1 - 2 e_2$ & $-e_1 - e_2$  \\
  \hline
$e_2$  & $-e_1 - e_2$ & $-2 e_1 + e_2$ \\
  \hline
\end{tabular}\\
\end{center}
\par
Thus,
$$
z^2 = (\alpha_1^2 - 2 \alpha_1 \alpha_2 - 2 \alpha_2^2) e_1 + (\alpha_2^2 - 2 \alpha_1 \alpha_2 - 2 \alpha_1^2) e_2.
$$
The equality $z^2 = z$ leads to the equations
$$\alpha_1 = \alpha_1^2 - 2 \alpha_1 \alpha_2 - 2 \alpha_2^2,~~\alpha_2 = \alpha_2^2 - 2 \alpha_1 \alpha_2 - 2 \alpha_1^2.$$
Subtracting the second equation from the first yields
$$
\alpha_1 - \alpha_2 = 3 (\alpha_1 - \alpha_2) (\alpha_1 + \alpha_2).
$$
It is not difficult to see that in this case the system of equations has only zero solution $\alpha_i=0$ for  $i=0,1,2$.
\medskip

Case 2. $\varepsilon(z) = 1$. In this case
 $\alpha_0 = 1 - \alpha_1 - \alpha_2$. Then $z = a_0 + \alpha_1 e_1 + \alpha_2 e_2$ and we get
$$
z^2 = (2 \alpha_2 + \alpha_1^2 - 2 \alpha_1 \alpha_2 - 2 \alpha_2^2) e_1 + (2 \alpha_1 + \alpha_2^2 - 2 \alpha_1 \alpha_2 - 2 \alpha_1^2) e_2.
$$
 From $z^2 = z$, we obtain the equations
$$\alpha_1 - 2 \alpha_2 = \alpha_1^2 - 2 \alpha_1 \alpha_2 - 2 \alpha_2^2,~~\alpha_2 - 2 \alpha_1 = \alpha_2^2 - 2 \alpha_1 \alpha_2 - 2 \alpha_1^2.$$
Subtracting the second from the first gives
$$
\alpha_1 - \alpha_2 =  (\alpha_1 - \alpha_2) (\alpha_1 + \alpha_2).
$$
If $\alpha_1 = \alpha_2$, then the system is equivalent to the equation $\alpha_1 = 3 \alpha_1^2$, which has only zero solution $(\alpha_1, \alpha_2) = (0,0)$. Thus, $\alpha_0=1$, and hence $z=a_0$ in this case. If $\alpha_1 \not = \alpha_2$, then the system has solutions $(\alpha_1, \alpha_2) = (1,0)$ or $(0,1)$. In this case $\alpha_0 = 0$, and hence $z=a_1$ or $a_2$.
\end{proof}

\medskip

Since $\R_4$ is disconnected and is disjoint union of its two trivial subquandles $\{a_0, a_2 \} $ and $\{a_1, a_3 \}$, we obtain
$$
I(\mathbb{Z}[\R_4]) \supseteq  I(\mathbb{Z}[\{a_0, a_2 \}]) \cup I(\mathbb{Z}[\{a_1, a_3 \}]).
$$
In fact, we have equality in this case.

\begin{proposition}\label{R4-idempotents}
$
I(\mathbb{Z}[\R_4]) = \Big\{ t\big(a_0 + \alpha (a_2 - a_0) \big) + (1-t) \big(a_1 + \beta (a_3 - a_1) \big)~|~t \in \{ 0,1\}, ~\alpha, \beta \in \mathbb{Z} \Big\}.
$
\end{proposition}

\begin{proof}
If $z = \alpha_0 a_0 + \alpha_1 a_1 + \alpha_2 a_2 + \alpha_3 a_3 \in \mathbb{Z}[\R_4]$, then a straightforward calculation gives
$$
z^2 = (\alpha_0^2 +  \alpha_0 \alpha_2 +  \alpha_1 \alpha_2 +  \alpha_2 \alpha_3) a_0 + (\alpha_1^2 + \alpha_0 \alpha_3 +  \alpha_1 \alpha_3 +  \alpha_2 \alpha_3) a_1 +
$$
$$
+(\alpha_2^2 + \alpha_0 \alpha_1 +  \alpha_0 \alpha_2 +  \alpha_0 \alpha_3) a_2 + (\alpha_3^2 +  \alpha_0 \alpha_1 +  \alpha_1 \alpha_2 +  \alpha_1 \alpha_3) a_3.
$$
The equality $z^2 = z$ holds if and only if the following system of equations
\begin{eqnarray*}
\alpha_0 &=& \alpha_0^2 +  \alpha_0 \alpha_2 +  \alpha_1 \alpha_2 +  \alpha_2 \alpha_3,\\
\alpha_1 &=& \alpha_1^2 +  \alpha_0 \alpha_3 +  \alpha_1 \alpha_3 +  \alpha_2 \alpha_3,\\
\alpha_2 &=& \alpha_2^2 +  \alpha_0 \alpha_1 +  \alpha_0 \alpha_2 +  \alpha_0 \alpha_3,\\
\alpha_3 &=& \alpha_3^2 +  \alpha_0 \alpha_1 +  \alpha_1 \alpha_2 +  \alpha_1 \alpha_3,
\end{eqnarray*}
has integral solutions. We use the observation from the beginning of this section and consider two cases:
\medskip

Case 1. $\varepsilon (z) = 0$, i.e. $z \in \Delta (\R_4)$. In this case
$$
z = \alpha e_1 + \beta e_2 + \gamma e_3 ~\mbox{for some}~ \alpha, \beta, \gamma \in \mathbb{Z},
$$
where $e_i = a_i - a_0$, $i = 1, 2, 3$. Using the multiplication table

\begin{center}
\begin{tabular}{|c||c|c|c|}
  \hline
 $\cdot$ & $e_1$ & $e_2$ & $e_3$ \\
   \hline
     \hline
$e_1$ & $e_1 - e_2 -e_3$ & $0$ & $e_1 - e_2 -e_3$ \\
  \hline
$e_2$  & $-2 e_2$ & $0$ & $-2 e_2$ \\
  \hline
$e_3$ & $-e_1 - e_2+e_3$ & $0$ & $-e_1 - e_2+e_3$ \\
  \hline
\end{tabular}\\
\end{center}
we obtain
$$
z^2 = (\alpha^2 - \gamma^2) e_1 - (\alpha^2 + 2 \alpha \beta + 2 \alpha \gamma + 2 \beta \gamma + \gamma^2)  e_2 + (-\alpha^2  + \gamma^2) e_3.
$$
The element $z$ is an idempotent if and only if
\begin{eqnarray*}
 \alpha &=& \alpha^2 - \gamma^2,\\
\beta &=& - (\alpha^2 + 2 \alpha \beta + 2 \alpha \gamma + 2 \beta \gamma + \gamma^2),\\
\gamma &=& -\alpha^2  + \gamma^2.
\end{eqnarray*}
Adding the first and third equations gives $\alpha + \gamma = 0$, i.e. $\gamma = -\alpha$. Then it follows from the system of equations that $\alpha = \beta = \gamma = 0$. Thus, $\Delta (\R_4)$ does not have non-zero idempotents.
\medskip

Case 2. $\varepsilon (z) = 1$, i.e. $z = a_0 + \delta$, where $\delta \in \Delta (\R_4)$ and
$$
\delta = \alpha e_1 + \beta e_2 + \gamma e_3 ~\mbox{for some}~ \alpha, \beta, \gamma \in \mathbb{Z}.
$$
We have $z^2 = a_0 + \delta a_0 + a_0 \delta + \delta^2$. Since $\Delta (\R_4)$ is a two-sided ideal, we have $\delta a_0, a_0 \delta \in \Delta (\R_4)$. Using the formulas
$$
e_1 a_0 = e_3,~~e_2 a_0 = e_2,~~e_3 a_0 = e_1,~~a_0 e_1 = e_2,~~ a_0 e_2 = 0,~~ a_0 e_3 = e_2,
$$
we obtain
$$
\delta a_0 = \alpha e_3 + \beta e_2 + \gamma e_1,~~a_0 \delta  = \alpha e_2 +  \gamma e_2.
$$
Using the expression for $\delta^2$ from Case 1 gives
$$
z^2 = a_0 + (\gamma + \alpha^2 - \gamma^2) e_1 + (\beta + \alpha + \gamma - \alpha^2 - \gamma^2 - 2 \alpha \gamma - 2 \alpha \beta - 2 \beta \gamma) e_2 + (\alpha - \alpha^2 + \gamma^2) e_3.
$$
Now $z$ is an idempotent if and only if the system of equations
\begin{eqnarray*}
\alpha &=& \gamma + \alpha^2 - \gamma^2,\\
0 &=&  \alpha + \gamma - \alpha^2 - \gamma^2 - 2 \alpha \beta - 2 \alpha \gamma - 2 \beta \gamma,\\
\gamma &=& \alpha - \alpha^2  + \gamma^2,
\end{eqnarray*}
has integral solutions. The first equation has the form
$$
(\alpha - \gamma) = (\alpha - \gamma) (\alpha + \gamma).
$$
Suppose that $\alpha = \gamma$, then the second equation has the form $0 = \alpha (1 - 2 \alpha - 2 \beta)$. If $\alpha = 0$, then for arbitrary $\beta$ we have the idempotent $z = a_0 + \beta (a_2 - a_0)$. If $\alpha \not= 0$, then the second equation does not have solutions.

Suppose that $\alpha \neq \gamma$, then $\gamma = 1 - \alpha$ and the second equation gives $\beta = 0$. Hence, for arbitrary $\alpha$ we have the idempotent $z = a_3 + \alpha(a_1 - a_3).$
\end{proof}

\begin{remark}
Note that $\T_3$, $\R_3$ and $\Cs(4)$ are, up to isomorphism,  all the quandles of order 3, which we have considered in this section. Since the number of quandles grow rapidly with order (see \cite[Table 1]{EFT}), for example, there are 7 quandles of order 4 and 22 quandles of order 5, computation of idempotents seems, in general, a challenging problem. Further, Proposition \ref{idempotents-r3} deals with a connected quandle whereas Propositions \ref{trivial-quandle-idempotents} and \ref{R4-idempotents} consider disconnected quandles. This makes us suspect that probably connected quandles have only trivial idempotents.
\end{remark}
\medskip


\section{Maximal quandles in quandle rings}\label{sec-max-quandles}

Recall that $\mq(R[Q])$ denotes the set of maximal quandles in the quandle ring $R[Q]$. Obviously, $\{0\}$ is a trivial quandle in $R[Q]$, called the {\it zero quandle}. Further, if some quandle $X$ in $R[Q]$ contains $\{0\}$, then $X = \{ 0 \}$. We begin by determining maximal quandles in the quandle ring $R[\T]$.

\begin{proposition}
$\mq(R[\T])=\{x_0 + \Delta_R(\T) \}$, where $x_0 \in \T$ is a fixed element.
\end{proposition}

\begin{proof}
By Proposition \ref{trivial-quandle-idempotents}, $x_0 + \Delta_R(\T)$ is the complete set of idempotents of $R[\T]$. It remains to prove the quandle axioms (Q2) and (Q3). Taking $z = x_0 + \delta$ and $w = x_0 + \delta'$, where $\delta, \delta' \in \Delta(T_n)$, we see that
$$
z w = (x_0 + \delta) (x_0 + \delta') = x_0 + \delta x_0 + x_0 \delta' + \delta \delta' = x_0 + \delta = z.
$$
Thus, the two axioms hold, and $x_0 + \Delta_R(\T)$ is, in fact, a trivial quandle.
\end{proof}

As a direct consequence of Proposition \ref{idempotents-r3}, we obtain

\begin{proposition}\label{max-quandles-r3}
$\mq(\mathbb{Z}[\R_3])= \{\R_3\}$.
\end{proposition}

Regarding $\R_4$, which is disconnected, we prove the following.

\begin{theorem}
The quandle ring $\mathbb{Z}[\R_4]$ contains a unique maximal quandle, namely,
$$
M = \Big\{ t\big(a_0 + \alpha (a_2 - a_0) \big) + (1-t) \big(a_1 + \beta (a_3 - a_1) \big)~|~t \in \{ 0,1\}, ~\alpha, \beta \in \mathbb{Z} \Big\}.
$$
\end{theorem}

\begin{proof}
By Proposition \ref{R4-idempotents}, $M$ is the complete set of idempotents in $\mathbb{Z}[\R_4]$. Therefore any quandle in $\mathbb{Z}[\R_4]$ must be contained in $M$. Thus, to prove the theorem, we need to show that $M$ is itself a quandle. We write  $M = M_1 \sqcup M_2$, where
\begin{eqnarray*}
M_1 &=& \big\{ a_0 + \alpha (a_2 - a_0)~|~\alpha \in \mathbb{Z} \big\},\\
M_2 &=& \big\{a_1 + \beta (a_3 - a_1)~|~ \beta \in \mathbb{Z} \big\}.
\end{eqnarray*}
A direct check shows that each $M_i$ is a trivial quandle. Further, for $u = a_0 + \alpha (a_2 - a_0) \in M_1$ and $v = a_1 + \beta (a_3 - a_1) \in M_2$, we have
$$uv= (a_0 + \alpha (a_2 - a_0))(a_1 + \beta (a_3 - a_1))=a_0 + (1 - \alpha)(a_2 - a_0) \in M_1$$
and
$$vu= (a_1 + \beta (a_3 - a_1))(a_0 + \alpha (a_2 - a_0))= a_1 + (1-\beta) (a_3 - a_1) \in M_2.$$
Thus, $M$ is closed under multiplication. The proof would be complete once we show that the map $S_u : M \to M$ given by $S_u (w) = w  u$ is an automorphism of $M$ for each $u \in M$. 
Suppose that $u \in M_1$. Then $S_u |_{M_1}$ is the identity automorphism of $M_1$. If $w = a_1 + \beta (a_3 - a_1) \in M_2$, then
$$
S_u (w) = a_1 + (1 - \beta) (a_3 - a_1).
$$
Thus, $S_u |_{M_2}$ is the automorphism of $M_2$ that is induced by the automorphism of the subquandle $\{ a_1, a_3 \}$ permuting the elements, and hence $S_u$ is an automorphism of $M$.
\par
Now suppose  that $u \in M_2$. A direct check shows that $S_u |_{M_2}$ is the identity automorphism. If $w = a_0 + \alpha (a_2 - a_0) \in M_1$, then
$$
S_u (w) =  a_0 + (1 - \alpha)(a_2 - a_0).
$$
Thus, $S_u |_{M_1}$ is the automorphism of $M_1$ induced by the automorphism of the quandle $\{ a_0, a_2 \}$ permuting the elements, and $S_u$ is an automorphism of $M$ in this case as well.
\end{proof}

\begin{theorem}
$\mq(\mathbb{Z}[\Cs(4)])= \{N_1, N_2 \}$, where
$$
N_1= \big\{ z, ~(1 - \beta) x +  \beta y~|~\beta \in \mathbb{Z} \big\},
$$
$$
N_2= \big\{\alpha x + \alpha y +  (1 - 2 \alpha) z~|~ \alpha \in \mathbb{Z} \big\}.
$$
\end{theorem}

\begin{proof}
Note that $N_1$ and $N_2$ are subsets of $I(\mathbb{Z}[\Cs(4)])$ which, by Proposition \ref{joyce-quandle},  is the complete set of idempotents in $\mathbb{Z}[\Cs(4)]$. A direct check shows that both $N_2$ and the set $\big\{(1 - \beta) x +  \beta y~|~\beta \in \mathbb{Z} \big\}$ are trivial quandles. Further, 
$$z((1 - \beta) x +  \beta y)=z$$
and 
$$((1 - \beta) x +  \beta y)z=\beta x +(1-\beta)y.$$
Thus, the map $S_z: N_1 \to N_1$ act by permuting the elements $x, y$ and fixing the element $z$, which shows that $N_1$ is also a quandle. It remains to show that $N_1$ and $N_2$ are maximal.
\par

For each $u \in I(\mathbb{Z}[\Cs(4)])$, consider the map $S_u : I(\mathbb{Z}[\Cs(4)]) \to I(\mathbb{Z}[\Cs(4)])$ given by $S_u(w)=wu$. If $u = (1 - \beta) x +  \beta y$, then a direct check shows that $S_u$ is the identity map on $I(\mathbb{Z}[\Cs(4)])$. And, if $u = \alpha x + \alpha y + (1 - 2 \alpha) z$, then $S_u |_{N_2}$ is the identity map. On the other hand, for elements of the form $w = (1 - \beta) x +  \beta y$, we have
$$
S_u (w) = ((1 - \beta) x +  \beta y ) (\alpha x + \alpha y + (1 - 2 \alpha) z )=(1 - \beta') x +  \beta' y,
$$
where $\beta'=1-\beta-2 \alpha+4 \alpha \beta$. Although it follows that $S_u(vw)=S_u(v)S_u(w)$ for all $v, w \in I(\mathbb{Z}[\Cs(4)])$, it turns out that $S_u$ is surjective if and only if $u=z$. In fact, given $(1 - \gamma) x +  \gamma y$, there exists $(1 - \beta) x +  \beta y$ such that $S_u((1 - \beta) x +  \beta y)=(1 - \gamma) x +  \gamma y$ if and only if $\beta= (\gamma+2 \alpha-1)/(4 \alpha-1)$. This equation admits an integral solution for each $\gamma$ if and only if $\alpha=0$. Thus, $N_1$ and $N_2$ are the only two maximal quandles in $\mathbb{Z}[\Cs(4)]$.

\end{proof}

\begin{theorem} \label{mod2-mq-R3}
$\mq(\mathbb{Z}_2[\R_3])= \Big\{\{ a_0 + a_1 + a_2 \},~ \R_3,~ \{ a_0 + a_1,  a_0 + a_2, a_1 + a_2 \} \Big\}$,\\ where  $\R_3 \cong \{ a_0 + a_1,  a_0 + a_2, a_1 + a_2 \}$.
\end{theorem}

\begin{proof}
By Proposition \ref{max-quandles-r3}, $\mq(\mathbb{Z}[\R_3]) =\{ \R_3\}$. We use the mod 2 reduction homomorphism $\varphi_2 : \mathbb{Z}[\R_3] \to \mathbb{Z}_2[\R_3]$ to determine $\mq(\mathbb{Z}_2[\R_3])$. The quandle ring $\mathbb{Z}_2[\R_3]$ contains 8 elements and a direct check shows that all its elements are idempotents. Denote
$$
z_{\varepsilon_0,\varepsilon_1,\varepsilon_2} = \varepsilon_0 a_0 + \varepsilon_1 a_1 + \varepsilon_2 a_2,~~~\varepsilon_i \in \{ 0, 1 \},
$$
and set $S_{\varepsilon_0,\varepsilon_1,\varepsilon_2}: \mathbb{Z}_2[\R_3] \to \mathbb{Z}_2[\R_3]$ be the right multiplication by the element $z_{\varepsilon_0,\varepsilon_1,\varepsilon_2}$. Then the maps $S_{1,0,0}$, $S_{0,1,0}$ and $S_{0,0,1}$ are automorphisms of order 2 since they are automorphisms of $\R_3$. We now determine actions of the other maps.

The maps $S_{1,1,0}$ acts by the rules:
$$
S_{1,1,0} (a_0) = a_0 + a_2,~~S_{1,1,0} (a_1) = a_2 + a_1,~~S_{1,1,0} (a_2) = a_1 + a_0,~~
S_{1,1,0} (a_0+ a_1) = a_0 + a_1,
$$
$$
S_{1,1,0} (a_0+ a_2) = a_1 + a_2,~~S_{1,1,0} (a_1+ a_2) = a_0 + a_2,~~
S_{1,1,0} (a_0 + a_1 + a_2) = 0.
$$

The maps $S_{1,0,1}$ acts by the rules:
$$
S_{1,0,1} (a_0) = a_0 + a_1,~~S_{1,0,1} (a_1) = a_2 + a_0,~~S_{1,0,1} (a_2) = a_1 + a_2,~~
S_{1,0,1} (a_0+ a_1) = a_1 + a_2,
$$
$$
S_{1,0,1} (a_0+ a_2) = a_0 + a_2,~~S_{1,0,1} (a_1+ a_2) = a_0 + a_1,~~
S_{1,0,1} (a_0 + a_1 + a_2) = 0.
$$

The maps $S_{0,1,1}$ acts by the rules:
$$
S_{0,1,1} (a_0) = a_2 + a_1,~~S_{0,1,1} (a_1) = a_1 + a_0,~~S_{0,1,1} (a_2) = a_0 + a_2,~~
S_{0,1,1} (a_0+ a_1) = a_0 + a_2,
$$
$$
S_{0,1,1} (a_0+ a_2) = a_0 + a_1,~~S_{0,1,1} (a_1+ a_2) = a_1 + a_2,~~
S_{0,1,1} (a_0 + a_1 + a_2) = 0.
$$

The maps $S_{1,1,1}$ acts by the rules:
$$
S_{1,1,1} (a_0) = a_0 + a_2 + a_1,~~S_{1,1,1} (a_1) = a_2 + a_1 + a_0,~~S_{1,1,1} (a_2) = a_1 + a_0 + a_2,~~S_{1,1,1} (a_0 + a_1) =0,
$$
$$
S_{1,1,1} (a_0 + a_2) =0,~~S_{1,1,1} (a_1 + a_2) =0,~~S_{1,1,1} (a_0 + a_1 + a_2) =a_0 + a_1 + a_2.
$$
Looking at the images of these maps, we see that the only possible quandles in $\mathbb{Z}_2[\R_3]$ are $\{ a_0 + a_1 + a_2 \}$, $\R_3$, and $\{ a_0 + a_1,  a_0 + a_2, a_1 + a_2 \}$, where $\R_3$ is clearly isomorphic to $ \{ a_0 + a_1,  a_0 + a_2, a_1 + a_2 \}$.
\end{proof}

An immediate consequence of Theorem \ref{mod2-mq-R3} and Proposition \ref{max-quandles-r3} is the following.

\begin{corollary}
The map $\mq(\mathbb{Z}[\R_3]) \to \mq(\mathbb{Z}_2[\R_3])$ induced by the mod 2 reduction homomorphism $\mathbb{Z}[\R_3] \to \mathbb{Z}_2[\R_3]$ is not surjective.
\end{corollary}
\medskip


\section{Automorphisms of quandle algebras}\label{sec-auto}

For a quandle $Q$ denote by $\Aut (R[Q])$ the group of $R$-algebra automorphisms of $R[Q]$, that is, ring automorphisms of $R[Q]$ that are $R$-linear. It is evident that $\Aut (Q) \leq \Aut (R[Q])$. Further, if $Q$ is a finite quandle with $n$ elements, then $\Aut (R[Q]) \leq \mathrm{GL}_n(R)$.

Note that any $\phi \in \Aut (R[Q])$ is defined by its action on elements of $Q$. Suppose that $Q = \{x_1, x_2, \ldots, x_n \}$. Then each $\phi(x_i)$ is an idempotent of $R[Q]$ and the quandle $\phi(Q)$ is isomorphic to $Q$. Using these facts we determine the automorphism groups of quandle algebras of some quandles of small orders.

\medskip

Let $\T_n = \{ x_1, x_2, \ldots, x_n \}$ be the $n$-element trivial quandle. We know that $\Aut (\T_n)$ is isomorphic to the symmetric group $\Sigma_n$. Since the group $\Aut (\mathbb{Z}[\T_1])$ is trivial, we assume that $n > 1$.  If $\phi \in \Aut (\mathbb{Z}[\T_n])$, then $\phi(\T_n)$ is an $n$-element trivial quandle and $\phi$ is an isomorphism of the $\mathbb{Z}$-module $\mathbb{Z}[\T_n]$. Since each $\phi(x_i)$ is an idempotent, by Proposition \ref{trivial-quandle-idempotents}, we have
$$
\phi(x_1) = x_1 + \alpha_{11} (x_2-x_1) + \alpha_{21} (x_3-x_1) + \cdots + \alpha_{n-1, 1} (x_n-x_1),
$$
$$
\phi(x_2) = x_1 + \alpha_{12} (x_2-x_1) + \alpha_{22} (x_3-x_1) + \cdots + \alpha_{n-1, 2} (x_n-x_1),
$$
$$
\vdots 
$$
$$
\phi(x_n) = x_1 + \alpha_{1n} (x_2-x_1) + \alpha_{2n} (x_3-x_1) + \cdots + \alpha_{n-1, n} (x_n-x_1),
$$
and the main problem is to find such integers $\alpha_{ij}$ such that the matrix $[\phi]$ has determinant $\pm 1$.

For the case $n=2$ we have

\begin{theorem}
$\Aut (\mathbb{Z}[\T_2]) \cong \mathbb{Z} \rtimes \mathbb{Z}_2.$
\end{theorem}

\begin{proof}
For any automorphism $\phi \in \Aut (\mathbb{Z}[\T_2])$, by the preceding discussion, we have
$$ \phi : \left\{ \begin{array}{ll} 
x_1 \longmapsto (1 - \alpha) x_1 + \alpha x_2, \\ 
x_2 \longmapsto (1 - \beta) x_1 + \beta x_2, 
\end{array} \right. 
$$
for some integers $\alpha$ and $\beta$. We first determine $\alpha$ and $\beta$ for which $\phi$ is an automorphism of the $\mathbb{Z}$-module $\mathbb{Z}[\T_2]$. For that to hold, if
$$
[\phi] =
\left(
  \begin{array}{cc}
1 - \alpha &1 - \beta \\
\alpha & \beta \\
  \end{array}
\right),
$$
then $\det([\phi]) = \beta - \alpha$ must be equal to $\pm 1$.

If $\det([\phi]) = 1$, then $\beta - \alpha = 1$ and
$$
A_{\alpha} :=[\phi] = 
\left(
  \begin{array}{cc}
1 - \alpha & -\alpha\\
 \alpha & 1 + \alpha \\
  \end{array}
\right).
$$

If $\det([\phi]) = -1$, then $\beta - \alpha = -1$ and
$$
B_{\alpha} :=[\phi] = 
\left(
  \begin{array}{cc}
1 - \alpha & 2 - \alpha \\
\alpha&  \alpha - 1 \\
  \end{array}
\right).
$$
A direct check shows that the automorphism $\phi$ corresponding to $A_{\alpha}, B_{\alpha}$ preserve the ring multiplication in $\mathbb{Z}[\T_2]$. Thus, we have
$$
\Aut(\mathbb{Z}[\T_2]) = \big\{ A_{\alpha}, B_{\alpha} ~| ~ \alpha \in \mathbb{Z} \big\}.
$$
It is easy to see that $A_0 = I$ is the identity matrix, and for arbitrary integers $\alpha, \beta$ the following formulas holds
$$
A_{\alpha} A_{\beta} = A_{\alpha + \beta},~~~B_{\alpha} B_{\beta} = A_{\alpha- \beta}.
$$
It follows from the first formula that $\{ A_{\alpha} ~|~ \alpha \in \mathbb{Z}\}$ is the infinite cyclic group with generator $A_1$. The second formula gives $B_{\beta} = A_{\beta -1}B_1 $, and hence $\mathrm{\Aut }(\mathbb{Z}[\T_2])$ is generated by $A_1$ and $B_1$. The matrix $B_1$ has order 2 and it is permutation of $x_1$ and $x_2$. Since $B_1 A_{\alpha} B_1 = A_{-\alpha}$, the subgroup $\langle A_1 \rangle$ is normal in $\mathrm{\Aut }(\mathbb{Z}[\T_2])$, and we have the desired result.
\end{proof}

\begin{theorem}
$\Aut (\mathbb{Z}[\Cs(4)]) \cong \mathbb{Z}_2$.
\end{theorem}

\begin{proof} 
Let $\phi \in \Aut (\mathbb{Z}[\Cs(4)])$. Since image of an idempotent under $\phi$ is an idempotent, by Proposition \ref{joyce-quandle}, we have
$$\{\phi(x), \phi(y), \phi(z)\} \subset \big\{ (1 - \beta) x +  \beta y~|~\beta \in \mathbb{Z} \big\} \cup \big\{\alpha x + \alpha y +  (1 - 2 \alpha) z~|~ \alpha \in \mathbb{Z} \big\}.$$ A direct check shows that the images of all the three generators cannot be of the same type,  else $\phi$ would not be a bijection. If
$$
 \phi : \left\{ \begin{array}{ll} 
x \longmapsto (1 - \alpha) x +  \alpha y, \\ 
y \longmapsto \beta x + \beta y +  (1 - 2 \beta) z,\\
z \longmapsto (1 - \gamma) x +  \gamma y,  
\end{array} \right. 
$$
then the relation $xz=y$ gives $1=2 \beta$, a contradiction. Similarly, if
$$
 \phi : \left\{ \begin{array}{ll} 
x \longmapsto \alpha x + \alpha y +  (1 - 2 \alpha) z, \\ 
y \longmapsto (1 - \beta) x +  \beta y,\\
z \longmapsto \gamma x + \gamma y +  (1 - 2 \gamma) z,  
\end{array} \right. 
$$
then the relation $yz=x$ gives $1=2 \alpha$, again a contradiction. Interchanging roles of $x$ and $y$, we see that $\phi(y)$ and $\phi(z)$ cannot be of the same type. Thus, only  $\phi(x)$ and $\phi(y)$ are idempotents of the same type. Arguments as above show that the only possibility is
$$
 \phi : \left\{ \begin{array}{ll} 
x \longmapsto (1 - \alpha) x +  \alpha y, \\ 
y \longmapsto (1 - \beta) x +  \beta y,\\
z \longmapsto \gamma  x + \gamma  y +  (1 - 2 \gamma ) z.
\end{array} \right. 
$$
Computing $\det([\phi])$ and equating to $\pm 1$ gives $(1-2 \gamma)(\beta-\alpha)= \pm 1$. This implies that $\gamma=0,1$ and $\beta=\alpha+\epsilon$, where $\epsilon= \pm1$.
\par
 If $\gamma=0$, then evaluating $\phi$ on the relation $xz=y$ gives $2 \alpha=1-\epsilon$. For $\epsilon=1$, we see that $[\phi]$ is the identity matrix. On the other hand, $\epsilon=-1$ gives
 $$
A:=[\phi]=
\left(
  \begin{array}{ccc}
0 & 1 & 0\\
1 & 0 & 0\\
0 & 0 & 1\\
  \end{array}
\right).
$$
In fact, $A$ is induced by the quandle automorphism $x \mapsto y$, $y \mapsto x$, $z \mapsto z$, which obviously preserve the ring multiplication.
\par

Similarly, if $\gamma=1$, then evaluating $\phi$ on the relation $xz=y$ gives $2 \alpha=1-\epsilon$. In this case, $\epsilon=1$ gives
$$
B_1:=[\phi]=
\left(
  \begin{array}{ccc}
1 & 0 & 1\\
0 & 1 & 1\\
0 & 0 & -1\\
  \end{array}
\right)
$$
and $\epsilon=-1$ gives
 $$
B_2:=[\phi]=
\left(
  \begin{array}{ccc}
0 & 1 & 1\\
1 & 0 & 1\\
0 & 0 & -1\\
  \end{array}
\right).
$$
An easy check shows that $B_i(x)B_i(z) \neq  B_i(y)$ for $i=1, 2$ although $xz=y$ in $\mathbb{Z}[\Cs(4)]$. Thus, only $A$ gives a desired automorphism, and  hence $\Aut (\mathbb{Z}[\Cs(4)]) \cong \mathbb{Z}_2$.
\end{proof}


\begin{proposition}
$\Aut (\mathbb{Z}[\R_3]) \cong \Sigma_3 \cong \mathrm{\Aut }(\R_3)$.
\end{proposition}

\begin{proof}
If $\phi \in \Aut (\mathbb{Z}[\R_3])$, then, by Proposition \ref{idempotents-r3},
$$\{\phi(a_0), \phi(a_1), \phi(a_2) \}= \{a_0, a_1, a_2 \}.$$
Thus, $\phi$ is represented by a permutation matrix, which lies in $\Sigma_3$. Conversely, a direct check shows that any $R$-module automorphism $\phi$ represented by a permutation matrix satisfies $\phi(a_i a_j)=\phi(a_i) \phi(a_j)$, and hence $\phi \in\Aut (\mathbb{Z}[\R_3])$.
\end{proof}

\begin{theorem}
$\Aut (\mathbb{Z}[\R_4]) \cong ( \mathbb{Z}_2 \times \mathbb{Z}_2) \rtimes \mathbb{Z}_2$.
\end{theorem}

\begin{proof}
Note that the quandle  $\R_4 = \{ a_0, a_1, a_2, a_3 \}$ is a disjoint union of trivial subquandles $\{a_0,a_2 \}$ and $\{a_1,a_3 \}$. Since an algebra automorphism maps idempotents to idempotents, it follows from Proposition \ref{R4-idempotents} that any $\phi \in \Aut (\mathbb{Z}[\R_4]) $ is of the form
$\phi(a_i)= (1 - \alpha_i) a_0 + \alpha_i a_2$ or $(1 - \alpha_i) a_1 + \alpha_i a_3$ for each $i$. It is clear that no three or more $\phi(a_i)$ can be of the same form else $\phi$ would not be a bijection. Thus, exactly two $\phi(a_i)$ are of one form and the remaining two of the other form. If $\phi(a_0)$ and $\phi(a_1)$ are of the same form, then evaluating $\phi$ on $a_0a_1=a_2$ gives a contradiction. Similar arguments show that $\phi(a_0)$ and $\phi(a_3)$ cannot be of the same form. Thus, we must have
$$
 \phi : \left\{ \begin{array}{ll} 
a_0 \longmapsto (1 - \alpha) a_0 + \alpha a_2, \\ 
a_1 \longmapsto (1 - \beta) a_1 + \beta a_3,\\
a_2 \longmapsto (1 - \gamma) a_0 + \gamma a_2,\\
a_3 \longmapsto (1 - \delta) a_1 + \delta a_3,
\end{array} \right. 
$$
or
$$
 \phi : \left\{ \begin{array}{ll} 
a_0 \longmapsto (1 - \alpha) a_1 + \alpha a_3, \\ 
a_1 \longmapsto (1 - \beta) a_0 + \beta a_2,\\
a_2 \longmapsto (1 - \gamma) a_1 + \gamma a_3,\\
a_3 \longmapsto (1 - \delta) a_0 + \delta a_2,
\end{array} \right. 
$$
for some integers $\alpha, \beta, \gamma, \delta$. One can check that the second automorphism can be obtained by composing the first with $\tau$, where $\tau$ is the algebra automorphism induced by the quandle automorphism of $\R_4$ given by
$$
 \tau : \left\{ \begin{array}{ll} 
a_0 \longmapsto a_1, \\ 
a_1 \longmapsto a_0,\\
a_2 \longmapsto a_3,\\
a_3 \longmapsto a_2.
\end{array} \right. 
$$
Thus, it is enough to consider $\phi$ to be of first type. If we write
$$
[\phi] =
\left(
  \begin{array}{cccc}
    1 - \alpha & 0 & 1 - \gamma & 0\\
    0 & 1 - \beta & 0 & 1-\delta \\
    \alpha & 0 & \gamma & 0 \\
    0 & \beta & 0 & \delta \\
  \end{array}
\right),
$$
then $ \det ([\phi]) = \pm 1$ implies that $(\gamma - \alpha) (\delta-\beta)=\pm 1$. Thus, we can write $\gamma=\alpha+ \epsilon_1$ and $\delta= \beta+ \epsilon_2$, where $\epsilon_i = \pm 1$. 
\par
Applying $\phi$ on the identity $a_0a_1=a_2$ gives $1-\epsilon_1=2 \alpha$. Thus, $\epsilon_1=1$ yields $\alpha=0$ and $\gamma=1$, whereas $\epsilon_1=-1$ yields $\alpha=1$ and $\gamma=0$. Similarly, applying $\phi$ on the identity $a_1a_0=a_3$ gives $1-\epsilon_2=2 \beta$. In this case, $\epsilon_2=1$ gives $\beta=0$ and $\delta=1$, whereas $\epsilon_2=-1$ gives $\beta=1$ and $\delta=0$. Thus, we obtain four matrices $\{I, A, B, AB \}$, where 
$$
A =
\left(
  \begin{array}{cccc}
    1  & 0 & 0 & 0\\
    0 & 0 & 0 & 1 \\
0 & 0 & 1 & 0 \\
    0 & 1 & 0 & 0 \\
  \end{array}
\right)
$$
and 
$$
B =
\left(
  \begin{array}{cccc}
    0  & 0 & 1 & 0\\
    0 & 1 & 0 & 0 \\
1 & 0 & 0 & 0 \\
    0 & 0 & 0 & 1 \\
  \end{array}
\right).
$$
A direct check shows that all the four $\mathbb{Z}$-module automorphisms corresponding to these matrices preserve the ring multiplication in $\mathbb{Z}[\R_4]$. Further, $A^2=B^2=I$, $AB=BA$ and $[\tau] A [\tau]=B$, where $[\tau]$ is the matrix of $\tau$. Thus, 
$\Aut (\mathbb{Z}[\R_4]) \cong \big(\langle A \rangle \times  \langle B \rangle \big) \rtimes \big\langle [\tau] \big\rangle \cong ( \mathbb{Z}_2 \times \mathbb{Z}_2) \rtimes \mathbb{Z}_2.$
\end{proof}

\begin{problem}
Compute automorphism groups of integral quandle rings of all trivial, dihedral and free quandles.
\end{problem}

\medskip

\section{Commutator width in quandle algebras}\label{commutator-width}
Let $Q$ be a quandle and $R$ a commutative and associative ring with unity. Define the \textit{commutator} of elements $u, v \in R[Q]$ as the element $$[u, v]=uv-vu.$$ Then the \textit{commutator subalgebra} $R[Q]'$ of $R[Q]$ is the $R$-algebra generated by the set of all commutators in $R[Q]$. If $Q$ is a commutative quandle, then the commutator subalgebra $R[Q]' =\{0\}$. Since $\varepsilon([u, v])=0$ for each commutator $[u, v] \in R[Q]'$, we obtain

\begin{lemma}\label{commutator-augmentation}
$R[Q]' \le \Delta_R(Q)$.
\end{lemma}

The equality in the preceding lemma does not hold in general. For example, the dihedral quandle $\R_3$ is commutative, and hence $\mathbb{Z}[\R_3]'=0$. On the other hand, $\Delta(\R_3) =\langle e_1,  e_2 \rangle \neq 0$. 

\par

We define the \textit{commutator length} $\cl(u)$ of an element $u \in R[Q]'$ as 
$$\cl(u)=\min \big\{n~|~ u= \sum_{i=1}^n \alpha_i [u_i,v_i], \textrm{where}~\alpha_i \in R,~u_i, v_i \in R[Q]\big\}.$$
The \textit{commutator width} $\cw(R[Q])$ is defined as 
$$\cw(R[Q])=\sup \big\{\cl(u)~|~u \in R[Q]'\big\}.$$

In the remainder of this section, we compute the commutator width of a few quandle rings.  We remark that the analogous problem of computation of  commutator width of free Lie rings \cite{Bardakov}, free metabelian Lie algebras \cite{Poroshenko} and absolutely free and free solvable Lie rings of finite rank \cite{Roman'kov} has been considered in the literature.
\par

It follows from the definition of commutator width that a quandle $Q$ is commutative if and only if  $\cw(R[Q])=0$. Consequently, we have $\cw(R[\R_3])=0$.
\par

We say that a quandle $Q$ is {\it strongly non-commutative} if for every pair of distinct elements $x, y \in Q$ there exist elements $a, b \in Q$ such that $ab = x$ and $ba = y$. Obviously, every strongly non-commutative quandle is non-commutative.

\begin{theorem}\label{commutator-width-general}
Let $Q$ be a strongly non-commutative quandle or a non-commutative quandle admitting a 2-transitive action by $\Aut(Q)$. Then the following hold:
\begin{enumerate}
\item $R[Q]'=\Delta_R(Q)$.
\item If $Q$ has order $n$, then $1 \le \cw(R[Q]) \le n-1$.
\end{enumerate}
\end{theorem}

\begin{proof}
In view of Lemma \ref{commutator-augmentation}, we only need to show that $\Delta_R(Q) \le R[Q]'$. Since $\Delta_R(Q)$ is generated as an $R$-module by elements of the form $x-y$, where $x, y \in Q$ are distinct elements, it suffices to show that each such element is a commutator. Suppose first that $Q$ is strongly non-commutative. Let $x, y \in Q$ be two distinct elements. Then, by definition of a strongly non-commutative quandle, there exist $a, b \in Q$ such that $x-y = ab - ba  = [a, b] \in R[Q]'$. Now, suppose that $Q$ is non-commutative. Then, there exist elements $c, d \in Q$ such that $cd \neq dc$. By 2-transitivity of $\Aut(Q)$ action on $Q$, there exists $\phi \in \Aut(Q)$ such that $\phi(cd)=x$ and  $\phi(dc)=y$. This gives $x-y=\phi(cd)-\phi(dc)=[\phi(c), \phi(d)] \in  R[Q]'$, which proves assertion (1).
\par 
Let $Q=\{x_0, x_1, \ldots, x_{n-1} \}$. Since $R[Q]'=\Delta_R(Q)$ by (1), any $u \in R[Q]'$ is of the form $u= \sum_{i=1}^{n-1} (x_i-x_0)$, where each $(x_i-x_0)$ is a commutator as shown in the proof of (1). Since $Q$ is non-commutative we obtain $1 \le \cw(R[Q]) \le n-1$.
\end{proof}

\begin{corollary}
Let $G$ be an elementary abelian $p$-group with $p>3$ and $\phi \in \Aut(G)$ act as multiplication by a non-trivial unit of $\mathbb{Z}_p$. Then $1 \le \cw\big(R[\Alex(G,\phi)]\big) \le |G|-1$.
\end{corollary}

\begin{proof}
The quandle $\Alex(G,\phi)$ is non-commutative for $p>3$ and admit a 2-transitive action by $\Aut(Q)$ \cite[Theorem 6.4]{BDS}.  
\end{proof}

We remark that a complete description of finite 2-transitive quandles has been given in a recent work of Bonatto \cite{Bonatto} by extending results of Vendramin \cite{Vendramin}. The following result shows that the bounds in Theorem \ref{commutator-width-general} are not sharp.

\begin{theorem}
The following statements hold:
\begin{enumerate}
\item If $\T$ is a trivial quandle, then $\cw(R[\T])=1$.
\item $\cw(R[\R_4]) =1$.
\item $\cw(R[\Cs(4)])=1$.
\end{enumerate}
\end{theorem}

\begin{proof}
By Theorem \ref{commutator-width-general}, $R[\T]'=\Delta_R(\T)$, and hence any element of $R[\T]'$ has the form $u= \sum_{i=1}^n \alpha_i (x_i-x_0)$. Taking $v= (1- \sum_{i=1}^n \alpha_i)x_0 +  \sum_{i=1}^n \alpha_i x_i$ and $w=x_0$, we see that
\begin{eqnarray*}
[v, w] &=& vw-wv\\
&=& \varepsilon(w) v -\varepsilon(v) w\\
&=& (1- \sum_{i=1}^n \alpha_i)x_0 +  \sum_{i=1}^n \alpha_i x_i -x_0\\
&=&  \sum_{i=1}^n \alpha_i (x_i-x_0)=u,
\end{eqnarray*}
and hence $\cw(R[\T])=1$, which proves (1).
\par

By Lemma \ref{commutator-augmentation}, $R[\R_4]' \le \Delta_R(\R_4)$. We know that $\Delta_R(\R_4)$ is generated by $\{e_1, e_2, e_3\}$, where $e_i=a_i-a_0$ for $i=1,2,3$. Let us find the commutators of the generators.
One can check that the elements 
$$
e_1 = [a_3, a_2],~~e_2=[a_2, a_0], ~~e_3 = -[a_2, a_1],
$$ 
lie in $R[\R_4]'$, i.e. $R[\R_4]' = \Delta_R(\R_4)$. 
We write an element $w=\alpha e_1+ \beta e_2 + \gamma e_3 \in \Delta_R(\R_4)$ as
$$
w=[a_2, \beta a_0 - \gamma a_1 - \alpha a_3].
$$ 
Thus, any element $w \in R[\R_4]'$ is a commutator, and hence $\cw(R[\R_4]) = 1$, which proves (2).
\par

The augmentation ideal $\Delta_R(\Cs(4))$ is generated by $e_1=y-x$ and $e_2=z-x$. Further, multiplication rules in $\Cs(4)$ show that $e_1=yx-xy$ and $e_2=zy-yz$, and hence $R[\Cs(4)]'=\Delta_R(\Cs(4))$.  Let $w=\gamma_1e_1+ \gamma_2e_2 \in \Delta_R(\Cs(4))$, where $\gamma_i \in R$. A direct check gives
$$e_1x=e_1,~~e_2x=e_2,~~xe_1=0,~~xe_2=e_1.$$
Now taking $u=x+ (\gamma_1+\gamma_2)e_1 + \gamma_2 e_2$ and $v=x$, we see that
$$[u, v]=uv-vu=w,$$
and hence $\cw(R[\Cs(4)])=1$, which establishes assertion (3).
\end{proof}

\begin{problem}
Compute commutator width of quandle algebras of dihedral and free quandles.
\end{problem}
\medskip

\section{Relation of quandle algebras with other algebras}\label{relations-other-algebras}
A group algebra is associative and for studying it we can use  methods of associative algebras. But the quandle algebras are not associative for non-trivial quandles. On the other hand, some classes of non associative algebras, for instance, alternative algebras, Jordan algebras and Lie algebras, are well studied. Thus, it is interesting to know whether quandle algebras belong to these classes of algebras. 
\par

We recall some definitions (see, for example \cite{ShShi}). Let $A$ be an algebra over a commutative and associative ring $R$ with unity. Then $A$ is called an {\it alternative algebra} if
$$
a^2 b = a (a b) ~\mbox{and}~ a b^2 = (a b) b~\mbox{for all}~a, b \in A;
$$
$A$ is called a {\it Jordan algebra} if it commutative and
$$
(a^2 b) a = a^2 (b a)~\mbox{for all}~a, b \in A;
$$
$A$ is called a {\it Lie algebra} if
$$
a^2 = 0 ~\mbox{and}~ (a b) c + (b c) a + (c a) b = 0~\mbox{for all}~a, b, c \in A;
$$
$A$ is called {\it power-associative} if every element of $A$ generates an associative subalgebra of $A$ \cite{Albert},
and $A$ is called an {\it elastic algebra} if 
$$
(x y) x = x (y x)~\textrm{for all}~ x, y \in A.
$$
For example, any commutative or associative algebra is elastic. In particular quandle algebras of trivial quandles are elastic being associative. It was proved in \cite[Proposition 7.3]{BPS} that quandle algebras of dihedral quandles are not power-associative over rings of characteristic
other than 2. The result was generalised to quandle algebras of all non-trivial quandles over rings of characteristic
other than 2 and 3 \cite[Theorem 3.4]{EFT}.

\begin{proposition}\label{elasticity-quandle-ring}
Let $Q$ be a non-trivial quandle. If $R$ is a ring of characteristic other than 2 and 3, then $R[Q]$ cannot be an alternative, an elastic or a Jordan algebra.
\end{proposition}

\begin{proof}
By \cite[Theorem 3.4]{EFT}, the quandle algebra $R(Q)$ is not power-associative. On the other hand,  any Jordan algebra is power-associative \cite[Chapter 2]{ShShi}, and by Artin's theorem any alternative algebra is also power-associative  \cite[Chapter 2]{ShShi}. Further, if $R[Q]$ is elastic, then $(xx)x=x(xx)$ and $(xx)(xx)= \big((xx)x\big)x$ for all $x \in R[Q]$, and hence $R[Q]$ is power-associative.
\end{proof}

\begin{remark}
If $Q$ is a quandle, then the third quandle axiom implies that $(x y) x = x (y x)$ for all $x, y \in Q$. Thus, all quandles satisfy the elasticity condition which, by Proposition \ref{elasticity-quandle-ring}, is in contrast to their quandle algebras.
\end{remark}

There are two natural constructions on any algebra $A = \langle A; +, \cdot \rangle$. Define an algebra $A^{(-)} = \langle A; +, \circ \rangle$ with multiplication
$$
x \circ y = x y - y x.
$$
If $A$ is an associative algebra, then $A^{(-)}$ is a Lie algebra. If $R$ contains $1/2$, then we define $A^{(+)} = \langle A; +, \odot \rangle$ with multiplication given by
$$
x \odot y = \frac{1}{2} (x y + y x).
$$
If $A$ is an associative algebra, then $A^{(+)}$ is a Jordan algebra. We conclude with the following result.

\begin{theorem}
Let $\T = \{ x_1, x_2, \ldots \}$ be a trivial quandle with more than one element, $A=R[\T]$, $L = A^{(-)}$ the corresponding  Lie algebra and $J = A^{(+)}$  the corresponding Jordan algebra. If $L^k$ is the subalgebra of $L$ generated by products of $k$ elements of $L$, then the following hold:
\begin{enumerate}
\item $L^2 = L^3$ and this algebra has a basis
$$
\{x_1 - x_2, x_2 - x_3, \ldots \}.
$$
In particular, if $T = T_n$ contains $n$ elements, then $L$ has rank $n-1$.

\item $(L^2)^2 = 0$, i.e. $L$ is a metabelian algebra.

\item $J^2 = J$.
\end{enumerate}
\end{theorem}

\begin{proof} The algebra $L^2$ is generated by the products $x_i \circ x_{j} = x_i - x_j$, $i < j$.
Denote $e_i = x_i - x_{i+1}$ and show that any product $x_i \circ x_{j}$ is a linear combination of $e_i$. Indeed, if $j = i+1$, then this product is $e_i$. If $j-i > 1$, then
$$
x_i - x_j = e_i + e_{i+1} + \cdots + e_{j-1}.
$$
On the other hand, it is easy to see that elements $e_1, e_2, \ldots$ are linearly independent. To determine $L^3$, we compute the products
$$
e_i \circ x_j =  e_i,
$$
$$
x_j \circ e_i = -e_i.
$$
Hence, $L^2 \subseteq L^3$ and assertion (1) follows.
\par

The algebra $(L^2)^2$ is generated by the products $e_i \circ e_j$. Straightforward calculation gives
$$
e_i \circ e_j = e_i e_j - e_j e_i = (x_i - x_{i+1}) (x_j - x_{j+1}) - (x_j - x_{j+1})  (x_i - x_{i+1}) = 0,
$$
which is assertion (2).
\par

For (3), since
$$
x_i \odot x_j = \frac{1}{2} (x_i x_j + x_j x_i) = \frac{1}{2} (x_i + x_j),
$$
it follows that $J^2$ contains elements $x_i = x_i \odot x_i$, and hence $J^2 = J$.
\end{proof}
\medskip

\begin{ack}
Bardakov is supported by the Russian Science Foundation grant  19-41-02005. Passi is thankful to Ashoka University, Sonipat, for making their facilities available. Singh acknowledges support from the Department of Science and Technology grants DST/INT/RUS/RSF/P-19 and DST/SJF/MSA-02/2018-19.
\end{ack}

\end{document}